\documentclass[11pt,reqno,tbtags]{amsart}
\usepackage[]{amsmath,amssymb,amsfonts,latexsym,amsthm,enumerate}
\usepackage{amssymb}
\usepackage{enumerate,graphicx,subfigure,bbm}
\usepackage{algorithm}

\numberwithin{equation}{section}

\newtheorem{maintheorem}{Theorem}

\newtheorem{theorem}{Theorem}[section]   

\newtheorem{lemma}[theorem]{Lemma}
\newtheorem{claim}[theorem]{Claim}
\newtheorem{proposition}[theorem]{Proposition}

\newcommand{\hmu}{\hat{\mu}}

\newcommand{\by}{\bar{Y}}
\newcommand{\mn}{\mathcal{M}_n}

\newskip\storeadskip
\newskip\storebdskip

\begin{document}

\title{Properties of Uniform Doubly Stochastic Matrices}

\author{Sourav Chatterjee}
\address{Courant Institute of Mathematical Sciences, New York University, 251 Mercer Street, New York, NY 10012.}
\email{sourav@cims.nyu.edu}
\thanks{Sourav Chatterjee's research was partially supported by NSF grant DMS-0707054 and a Sloan Research Fellowship}
\author{Persi Diaconis}
\address{Department of Mathematics, Stanford University,  Stanford, CA 94305}
\email{diaconis@math.stanford.edu}
\author{Allan Sly}
\address{Theory Group, Microsoft Research, One Microsoft Way, Redmond, WA 98052}
\email{allansly@microsoft.com}
\maketitle

\begin{abstract}
We investigate the properties of uniform doubly stochastic random matrices, that is non-negative matrices conditioned to have their rows and columns sum to 1.  The rescaled marginal distributions are shown to converge to exponential distributions and indeed even large sub-matrices of side-length $o(n^{1/2-\epsilon})$ behave like independent exponentials.  We determine the limiting empirical distribution of the singular values the the matrix. Finally the mixing time of the associated Markov chains is shown to be exactly 2 with high probability.
\end{abstract}
\vspace{1 cm}

Random matrices have become a central area of focus for modern probability theory and numerous models have been intensely studied including Wigner, Wishart, GOE and GUE matrices~\cite{AGZ:09}.  In this paper we study a model for which much less is known, namely uniformly chosen entries of the set of doubly stochastic matrices (called Uniformly Distributed Stochastic Matrices).  The  Birkhoff polytope is an $(n-1)^2$ dimensional polytope in $\mathbb{R}^{n^2}$ constituting the set of doubly stochastic matrices and is the convex hull of the permutation matrices (see e.g. \cite{Stanley:97}).   While its extreme points are  sparse matrices we shall see that typical entries chosen according to the uniform distribution are by contrast very dense.  Little is known about the properties of uniformly distributed stochastic matrices as they fall outside the scope of techniques from the usual random matrix theory, however, important recent progress has been made by Barvinok and Hartigan.

We will let $X=(X_{ij})_{i,j=1,\ldots,n}$ denote a uniform doubly stochastic matrix.  By symmetry its rows and columns are exchangeable and all its entries have the same marginal distribution.  It is natural then to ask what is the limiting distribution of $n X_{11}$, the first entry rescaled to have mean 1.  In our first result we determine that the rescaled marginal distribution converges to an exponential random variable of mean 1.

\begin{maintheorem}\label{t:mainThm}
With $X=(X_{ij})_{i,j=1,\ldots,n}$ a uniformly chosen doubly stochastic matrix we have that,
\[
n X_{11} \stackrel{d}{\rightarrow} \exp(1)
\]
as $n\rightarrow\infty$ where the convergence is in total variation distance.  Further, for any $\epsilon>0$,
\[
d_{\mathrm{tv}}(n X_{11}, \exp(1)) = O(n^{-1/2+\epsilon}).
\]
\end{maintheorem}

A natural extension to this question is to ask about the joint distribution for a collection of  several entries.  It can be shown using the same approach that finite collections of random variables converge to independent exponentials with mean 1.  This convergence holds not just in distribution but also in total-variation distance and its moments converge to the moments of independent exponentials (see Section~\ref{s:moments}).  We believe that in many ways uniformly distributed stochastic matrices behave much like matrices of independent.  For example the largest entry of the matrix is at most $(2+o(1))\frac1n \log n$ with high probability,

\begin{maintheorem}\label{t:maxEntry}
For any $\epsilon>0$,
\[
P\left(\max_{1\leq i,j\leq n} n X_{ij} > (2+\epsilon)\log n \right ) \rightarrow 0,
\]
as $n\to\infty$.
\end{maintheorem}

Another question one may ask is the limiting distribution of the singular values of $\bar{X}=n^{1/2}(X-EX)$.  Denote these by $0\leq\sigma_1(\bar{X})\leq\ldots\leq\sigma_n(\bar{X})$.  Letting $\mu$ denote the measure on $[0,2]$ with density
\[
\frac{1}{\pi }\sqrt{4-x^2}
\]
we have the following result.
\begin{maintheorem}\label{t:singular}
The limiting empirical singular value distribution of $\bar{X}$ is given by
\[
\sum_{i=1}^n \delta_{\sigma_i(\bar{X})} \rightarrow \mu
\]
where the convergence is in the weak topology, in probability as $n\to\infty$.
\end{maintheorem}
We conjecture that the empirical spectral distribution converges to the circular law.

One natural question is to ask how large a sub-matrix can one take so that the entries are still asymptotically independent.  This problem was studied in the context of the random orthogonal matrix~\cite{Jiang:06} where it was shown that an $k\times k$ sub-matrix is asymptotically distributed as independent normal random variables in total variation provided $k=o(n^{\frac12})$ answering a question of the second author~\cite{DEL:92}. In~\cite{Jiang:06} it is further shown that order $n/\log n$ entries simultaneously converge if weaker topologies are used.   Here we show that for sub-matrices of uniformly distributed stochastic matrices of size almost $n^{1/2}$ the entries are asymptotically independent.

\begin{maintheorem}\label{t:convergenceBox}
Let $V$ denote the projection of a uniformly distributed stochastic matrix onto the $k\times k$-sub-matrix of its first $k$ rows and columns and let $\Delta$ be a $k\times k$ matrix of independent mean one exponential random variables.
When $k=O(\frac{\sqrt{n}}{\log n})$  the rescaled law of $V$ converges to $\Delta$,
\[
d_{\mathrm{tv}}(n V,\Delta) \to 0
\]
as $n\to\infty$  where $d_{\mathrm{tv}}$ denotes the total variation distance.
\end{maintheorem}

Unlike most other classes of random matrices, uniformly distributed stochastic matrices are of course stochastic which raises the question of the properties of the associated Markov chains.  For any doubly stochastic Markov transition kernel the stationary distribution is the uniform distribution.  For a uniform stochastic (but not necessarily doubly stochastic) matrix, that is a uniformly chosen Markov chain, the mixing time is two asymptotically almost surely~\cite{AldDia:86}.  We show that this holds also for uniformly chosen doubly stochastic random matrices.

\begin{maintheorem}\label{t:mixing}
The mixing time of the Markov chain given by a uniform double stochastic matrix is with high probability~$2$.
\end{maintheorem}

In Section~\ref{s:back} we give background and history for  the Birkhoff polytope. In Section~\ref{s:marginal} we give the proofs  of Theorems~\ref{t:mainThm} and ~\ref{t:convergenceBox}.  Then in Section~\ref{s:other} we begin by studying polytopes of matrices with non-constant row sums.  By establishing that the volumes of the polytopes are maximized when the row and column sums are equal, we get strong control over the distribution of a row in a uniformly distributed stochastic matrix through which we can bound the tails of the marginal distributions establishing convergence of the moments and Theorem~\ref{t:maxEntry}.  Finally, knowing that the entries are not too large allows us to show strong concentration for the entries of $X^2$ which guarantees that the mixing time is 2.

\section{Background}\label{s:back}

This section gives background and references for four topics that motivate our work: the Birkhoff polytope, prior distributions on Markov chains, limit theorems for entries of large random matrices in classical compact groups and contingency tables with fixed row and column sums

\subsection{The Birkhoff Polytope}

The set $\mn$ of $n \times n$ doubly stochastic matrices is known as the Birkhoff polytope, the bistochastic polytope and the assignment polytope.  It is a basic object of study in operations research because of its appearance as the feasible set for the assignment problem.  Given a cost matrix $C_{ij}$ this asks for a permutation $\sigma$ minimizing $\sum_i C_{i\sigma(i)}$.  This is the same problem as minimizing $\sum_{ij} C_{ij}M_{ij}$ for $M\in \mn$ because of Birkhoff's Theorem: the permutation matrices are the extreme points of $\mn$. A thorough treatment of the assignment problem is in \cite{LovPlu:86}.

Because of this connection, the structure of $\mn$ has been intensively studied.  Two permutations $\sigma,\varsigma$ are adjacent on $\mn$ if and only if $\sigma \varsigma^{-1}$ is a cycle (see \cite{EKK:84} page 214).  The diameter (the maximum distance between two vertices on the skeleton) of $\mn$ is two~\cite{EKK:84}.  The face structure of $\mn$ is described in \cite{BilSar:96}.  Finding a closed form expression for the volume of $\mn$ is a well known open problem.  The volume is a rational number and in  known for $n\leq 14$ (see \cite{LLY:09} and references therein).  The combinatorics suggest a simple probability problem: what is the mixing time of the nearest neighbor random walk on vertices of $\mn$?  Pak~\cite{Pak:00} showed that it is two.

Birkhoff's characterisation of the extreme points is ``equivalent'' to other basic theorems in combinatorics such as Kontg's Lemma, Hall's Marriage Theorem and the Max-flow Min-Cut Theorem.  A splendid account of these connections is in \cite{LovPlu:86}.

There are other polytopes with similarly nice descriptions.  For example, the symmetric doubly stochastic matrices have extreme points $\frac12(A_\sigma+A_\sigma^T)$ with $A_\sigma$ the permutation matrix of $\sigma$~\cite{Cruse:75,Sackov:75}.  Perhaps the methods and results of our paper can be used to study the behavior of a randomly chosen point in these polytopes.  The properties of the random tri-diagonal doubly stochastic matrices are thoroughly studied in~\cite{DiaMaw:10}.

\subsection{Statistical Analysis of Markov Chains}

Our original motivation for this work comes from the statistical analysis of a Markov chain on $\{1,2,\ldots,n\}$ with unknown transition matrix $(X_{ij})\in Q_n$ ($Q_n$ the set of stochastic matrices).  One observes a run $R_0,R_1,\ldots,R_N$ and is requried to estimate $(X_{ij})$.  A Bayesian approach to this problem starts with a prior distribution on $Q_n$.  The classical Bayesian approach using, conjugate priors, sets each row to be an independent Dirichlet distribution.  One natural choice has each Dirichlet distribution as uniform on the $n$-simplex.  This gives the measure studied below.  For background and references see \cite{Martin:75, DiaFre:87, Zabell:95}.

Recent developments put priors on natural subclasses of Markov chains.  For example~\cite{DiaRol:06, BCP:09} develop and apply priors for reversible Markov chains and~\cite{BacPan:10} develop priors for higher order Markov chains.

It is natural to consider priors on the space of Markov chains with a fixed (known) stationary distribution.  This is again a connected convex set.  Perhaps the most natural example is the uniform distribution on $\{1,2,\ldots,n\}$.  Now the set of transition matrices is the Birkhoff polytope  and the uniform distribution is a natural prior.  Understanding the uniform distribution for large $n$ leads to the topics in this paper.

Knowing about Birkhoff's Theorem it  is also natural to study the prior measure on $\mn$ resulting from a uniform combination of extreme points.  Thus if $A_\sigma$ is the permutation matrix corresponding to $\sigma$ and $\{X_\sigma\}$ is a uniform point of the $n!$-simplex then $M=\sum_{\sigma \in S_n} A_\sigma X_\sigma$ is a uniform combination of extreme points.  This distribution was proposed and studied in \cite{MelPet:95} as a way to put a prior on the parameters of an $n\times n$-contingency table with known uniform margins.  The following result suggests this is a strange distribution, sharply concentrated about the matrix with all entries $1/n$.
\begin{proposition}
Let $M\in \mn$ be a uniform convex combination of extreme points.  Then
\[
E \sum_{ij} |M_{ij}-\frac1n| \leq n \sqrt{\frac{n-1}{n!+1}}
\]
\end{proposition}
\begin{proof}
The distribution of $M_{11}$ is given by $\hbox{Beta}(a,b)$ distribution with $a=(n-1)!$ and $b=(n-1)(n-1)!$ which has mean $a/(a+b)=1/n$ and variance $ab/(a+b)^2(a+b+1)=(n-1)/n^2(n!+1)$.  Then by the symmetry of the entries
\[
E \sum_{ij} |M_{ij}-\frac1n| = n^2E |M_{11}-\frac1n| \leq n^2\sqrt{\hbox{Var} M_{11}} \leq n \sqrt{\frac{n-1}{n!+1}}.
\]
\end{proof}
Of course, this prior is absolutely continuous with respect to the uniform distribution and a sufficiently large amount of data will swamp the prior (although this may be prohibitive large when $n$ is large).

A variety of measures on the stochastic matrices were studied in the subject of ``random random walks''~\cite{Hildebrand:05}.  This area was initiated with a theorem of Aldous and Diaconis \cite{AldDia:86}.  If an $n\times n$ stochastic matrix is chosen by making the rows uniform on the $n$-simplex the expected time to stationarity is small, indeed two steps suffice (but one does not).  This suggests that this models does not capture the essential features of real Markov chains which are usually ``local''.  Much of the work thus restricts attention to random walks on finite groups $G$ (see~\cite{Hildebrand:05} for more details).

Our discussion leaves many points untouched.  To generate points from the uniform distribution on $\mn$ we use a basic ``Gibbs sampling algorithm'': pick a pair of distinct rows and a pair of distinct columns at random.  These intersect in a $2\times 2$ matrix
$A=\left(\begin{array}{cc} a  & b \\ c & d \\ \end{array} \right)$.  This is replaced by $\left(\begin{array}{cc} a'  & b' \\ c' & d' \\ \end{array} \right)$ chosen uniformly on the set of matrices with the same row and column sums as $A$.  This is easy to do choosing $a'$ uniformly from the relevant range.  
We would like to understand the running time of this algorithm.  A host of other algorithms for uniform choice in a compact set is in \cite{AndDia:07}.

The posterior distribution on $\mn$ after observing the Markov chain of length $N$ is proportional to $\prod_{i,j} x_{ij}^{N(i,j)}$ where $N(i,j)$ is the number of observed transitions from $i$ to $j$ in the run.  How do such measures behave?  Our work suggests a heuristic: the measures should behave like product Dirichlet distributions.  The ith row having density proportional to $\prod_{j} x_{i,j}^{N(i,j)}$.  The known properties of the Dirichlet distribution now make basic questions accessible.  For example, the Bayes estimate of the transition matrix is easy to compute.

\subsection{Elements of Random Matrices}

The present paper has many points of contact with the ongoing study of the behavior of entries of a uniformly chosen random matrix in one of the classical compact groups $O_n$ or $U_n$.  These problems we originally studied to understand the `equivalence of ensembles' in statistical mechanics. Indeed, the first row of a random matrix in $O_n$ is uniformly distributed on the $n$-sphere--the micro-canonical ensemble.  The entries multiplied by $\sqrt{n}$ are approximately independent standard normal--the canonical ensemble.  This is an early theorem of Borel; see \cite{DiaFre:87} for a historical review, sharp statements and pointers to the work of L\'{e}vy and others.  Later these theorems were extended and used to prove sharp finite forms of de Finetti's theorems and many extensions \cite{DEL:92}.

For $M$ chosen uniformly on $U_n$, the entries multiplied by $\sqrt{n}$ are approximately independent standard complex normal.  This has been proved in various sense.  For example \cite{Jiang:06} shows that an $m\times m$ block is close to normal in total variation if $m=o(\sqrt{n})$.  For other topologies \cite{Jiang:09} shows indepdent normal behaviour persists for $m=o(n/\log n)$.  Other global features, such as the maximum entry \cite{Jiang:05}, traces of powers of $M$ \cite{DiaSha:94,DiaEva:01} and arbitrary linear combinations of the entries \cite{ADN:03} behave like normals as well.  Of course there are differences.  The eigenvalues of a random element of $U_n$ lie on the unit circle while the eigenvalues of independent normals fill out the disk uniformly. For refinements, see \cite{Meckes:08,Mezzadri:07}.

Yuval Peres suggested that these results may have a close connection to the Birkhoff polytope.  Let $M$ be uniform in $U_n$ and set $N_{ij}=|M_{ij}|^2$. Then $N$ is doubly stochastic with entries approximately independent  and exactly exponentially distributed.  While we show in Section~\ref{s:nonConstant} that these distributions are not the same it seems likely that they share many properties.

Classical results for equivalence of ensembles show equivalence of micro-canonical and canonical ensembles which result from fixing low dimensional sufficient statistics.  The results above, and in the present paper, show that equivalences of various sorts persist after conditioning on high dimensional statistics: If $\{E_{ij}\}$ is a matrix of independent exponentials, the conditional distribution given that all the row and column sums are equal to one is uniform on $\mn$. More background on equivalence of ensembles can be found in \cite{Zabell:95} and \cite{Lanford:73}.

\subsection{Magic squares and contingency tables}

There is a close connection between the Birkhoff polytope $\mn$ and $MS(n,c)$ the set of $n\times n$ matrices with non-negative integer entires and all row and column sums equal to $c$.  Elements of $MS(n,c)$ are called magic squares in the enumerative literature.  It is known that $|MS(n,c)|$ is a polynomial in $c$ of degree $(n-1)^2$.  The leading coefficient of this polynomial is a simple multiple of the volume of $\mn$~\cite{Stanley:97}. See also \cite{DiaGam:04}.

Generalizing, the set of $m\times n$ matrices with non-negative entries and fixed row and column sums is intensively studied both in combinatorics and statistics where they are called contingency tables.  It is known that exact enumerations of the size of this set is $\#P$-complete even when $n=2$.  A host of techniques for approximate counting and random generation have been developed as well as a remarkable collection of asymptotic formulae.  See \cite{DiaGan:93} and \cite{Barvinok:09} for surveys.

Questions of the properties of random contingency tables or randomly chosen points in polytopes are closely connected to the problem of estimating the volume of the polytopes.  Important recent work by Barvinok and Hartigan has given asymptotic formulas for the number of contingency tables and the volumes of polytopes of such  matrices~\cite{BarHar:09a,BarHar:09b,Barvinok:10} as well as the closely related problem of the number of graphs with a given degree sequence~\cite{BarHar:10}.  A central idea in their analysis is the maximum entropy distribution which for the Birkhoff polytopes corresponds to independent exponentials for the vertices of the matrix.  This maximum entropy distribution provides a good approximation to the distribution yielding (after much work)  an asymptotic calculation of the volume.

Beyond asymptotic volume calculations Barvinok~\cite{Barvinok:09} also asked the question of ``what does a random contingency table look like''? In~\cite{Barvinok:10} a precise sense was given to the statement that ``in many respects a random matrix behaves as a matrix X of independent geometric random variables'', a direction pursued independently in this paper.  One result of this equivalence given in~\cite{Barvinok:09} is that the sum of large subsets of the entries of such contingency tables are concentrated around their expectation given under the maximum entropy distribution.
Barvinok~\cite{Barvinok:10} posed the natural question of determining the marginals of the entries of such random matrices. In the case of doubly stochastic matrices we answer this question determining that they are asymptotically independent exponentials.

\section{Marginals of Uniform Doubly Stochastic Matrices}\label{s:marginal}
Let $X=(X_{ij})_{i,j=1,\ldots,n}$ be a uniform doubly stochastic matrix, that is chosen uniformly from the Birkhoff polytope. Since the sum of the rows and columns add to 1, it satisfies $2n-1$ linear constraints and the matrix is determined by the $(n-1)^2$ entries $(X_{ij})_{i,j=1,\ldots,n-1}$.  Let $\Gamma:\mathbbm{R}^{(n-1)^2} \to \mathbbm{R}^{n^2}$ denote the function
\[
\Gamma(X)=\Gamma(X)_{ij}=\begin{cases}
X_{ij} &1\leq i,j\leq n-1,\\
1-\sum_{k=1}^{n-1} X_{ik} &1\leq i\leq n-1,j=n\\
1-\sum_{k=1}^{n-1} X_{kj} &1\leq j\leq n-1,i=n\\
1-\sum_{l=1}^{n-1} (1-\sum_{k=1}^n X_{kl}) &i=j=n \, .
\end{cases}
\]
Let $\Phi:\mathbbm{R}^{(n-1)^2} \to \mathbbm{R}^{n^2}$ be the projection $X\mapsto (X_{ij})_{1\leq i,j \leq n-1}$.  By an abuse of notation we will also use $\Gamma$ as a function from $\mathbbm{R}^{n^2}$ to itself by $\Gamma(\Phi(X))$.  Then the doubly stochastic matrices correspond to the $(n-1)\times(n-1)$-matrices in the set
\[
S_n=\left\{(x_{ij})_{i,j=1,\ldots,n-1} \in [0,1]^{(n-1)^2}: \min_{1\leq i,j \leq n} x_{ij} - \Gamma(x)_{ij}  \geq 0 \right \}.
\]
The distribution of $(X_{ij})_{i,j=1,\ldots,n-1}$ is given by the uniform distribution on $S_n$.
Let $Z_n$ denote the volume of $S_n$, that is
\[
Z_n = \int_{[0,1]^{(n-1)^2}} I(x\in S_n) dx
\]
where $I$ denotes the indicator function.
Canfield and McKay \cite{CanMcK:07} showed that asymptotically the volume of the Birkhoff polytope (in units of basic cells of the lattice which is equivalent to our usage) is
\begin{equation}\label{e:ZnAsymptotic}
Z_n=\frac1{n^{n-1}} \cdot \frac1{(2\pi)^{n-1/2}n^{(n-1)^2}}\exp\Big(\frac13+n^2+o(1)\Big).
\end{equation}
Also define
\[
\mathcal{D}_n = \left\{(y_{ij})_{i,j=1,\ldots,n} \in \mathbbm{R}^{n^2}: \Phi( \tfrac1n y) \in S_n,  \min_{i,j} (y - \Gamma(\tfrac1n y))_{ij}\geq 0 \right \}.
\]
As we observed in the introduction, the uniformly distributed stochastic matrix shares many properties with matrices of independent exponentials so let us define $(Y_{ij})_{1\leq i,j \leq n}$ as a matrix of iid exponential mean 1 random variables.
\begin{lemma}\label{l:ProbDn}
Conditional on $Y\in \mathcal{D}_n$ we have that $\tfrac1n (Y_{ij})_{1\leq i, j \leq n-1}$ is uniform on $S_n$.
Further, for large $n$ we have that,
\begin{equation}
P(Y\in \mathcal{D}_n) \geq n^{-4n}.
\end{equation}
\end{lemma}

\begin{proof}
Let $\mathcal{W}$ be the product of the intervals $\mathcal{W}= \prod_{1\leq i,j \leq n} I_{ij}$ where
\[
I_{ij}= \begin{cases} [0,\infty) &\ \hbox{if } \max\{i,j\}=n\\
\{0\} & \ \hbox{o.w.}
\end{cases}
\]
Then for each fixed $\by \in S_n$ the set $\{ Y\in \mathcal{D}_n: (\tfrac1n Y)_{i,j=1,\ldots,n-1} = \by\}$ is  $n\Gamma(\by)+\mathcal{W}$.  Since the density of $Y$ depends only on $\sum_{ij} Y_{ij}$ and since $\sum_{ij} n\Gamma(\by)_{ij} \equiv n^2$ it follows that $\Gamma(\tfrac1n Y)$ is uniform on $S_n$.   Now
\begin{align}\label{e:probDn}
P(Y\in \mathcal{D}_n) &= \int_{\mathbbm{R}^{n^2}}  \exp\left(-\sum_{i=1}^n \sum_{j=1}^n y_{ij} \right)I\left(Y\in \mathcal{D}_n\right) dy_{11}\ldots dy_{nn}\nonumber\\
&= \int_{n\mathcal{S}_n} \int_{\mathbbm{R}^{2n-1}} \exp\left(-\sum_{i=1}^n \sum_{j=1}^n n(\Gamma(\tfrac1n Y)_{ij} - [y_{ij} - n(\Gamma(\tfrac1n Y)_{ij})] \right)\\
&\qquad \cdot I\left(\min_{i,j} y_{ij} -\Gamma(Y)_{ij} \geq 0\right) dy_{11}\ldots dy_{nn}\nonumber\\
&= \hbox{Vol}_{n^2}(n\mathcal{S}_n)\exp(-n^2)\\
& \qquad\cdot\int_{[0,\infty)^{2n-1}}\int_{\mathbbm{R}^{2n-1}} \exp\left(-\sum_{i=1}^n y_{in} -\sum_{j=1}^{n-1} y_{in} \right)dy_{1n}\ldots dy_{nn} dy_{n1}\ldots dy_{n,n-1}\\
&=\hbox{Vol}_{n^2}(n\mathcal{S}_n)\exp(-n^2).
\end{align}
Combining equations \eqref{e:ZnAsymptotic}, \eqref{e:probDn}, \eqref{e:maxEntryRation} we have that
\begin{align}\label{e:probDn}
P(Y\in \mathcal{D}_n) &=\frac1{n^{n-1}} \cdot \frac1{(2\pi)^{n-1/2}n^{(n-1)^2}}\exp\Big(\frac13+n^2+o(1)\Big) n^{n^2}\exp(-n^2) \geq n^{-4n},
\end{align}
for large $n$.
\end{proof}

In particular this means for $X$ uniform on $\mn$, for any measurable set $\mathcal B\subset \mathbb{R}^{(n-1)^2}$, by equation \eqref{e:probDn} we have that
\begin{equation}\label{e:relativeProbabilities}
P(X\in \mathcal B) \leq  n^{4n} P(\Phi(Y) \in \mathcal B).
\end{equation}
This equation is only meaningful when $P(\Phi(Y) \in \mathcal B)\leq n^{4n}$. However, for a number of important large deviation events we can effectively translate results about $Y$ to results about $X$.  In particular using the exchangeability of the entries of $X$ we can establish the asymptotic marginal distribution of the entries of the $X$ given in Theorem~\ref{t:mainThm}.

\begin{proof}[Proof of Theorem~\ref{t:mainThm}]
Let $\mathcal{A}$ be a measurable subset of $[0,\infty)$.  By the Azuma–--Hoeffding inequality
\[
P\left( \left|\frac1{n(n-1)} \sum_{i=1}^{n-1}\sum_{j=1}^{n-1} I(n Y_{ij} \in \mathcal{A} ) - P(Y_{11}\in \mathcal{A})  \right | > \frac12 n^{-1/2+\epsilon} \right)\leq \exp(-cn^{1+2\epsilon}).
\]
Then by equation \eqref{e:relativeProbabilities} we have that,
\[
P\left( \left|\frac1{n(n-1)} \sum_{i=1}^{n-1}\sum_{j=1}^{n-1} I(n X_{ij} \in \mathcal{A} ) - P(Y_{11}\in \mathcal{A})  \right | > \frac12 n^{-1/2+\epsilon} \right) \leq n^{4n} \exp(-cn^2) \leq \exp(-c' n^2)
\]
and so since the entries of $X$ are exchangeable,
\[
 \left|P(n X_{11} \in \mathcal{A} ) - P(Y_{11}\in\mathcal{A}) \right | < n^{-1/2+\epsilon} +  \exp(-c' n^2).
\]
As this holds uniformly over all $\mathcal{A}$ it follows that $d_{\mathrm{tv}}(X_{11},Y_{11})< n^{-1/2+\epsilon}$ for large $n$ which establishes the result.
\end{proof}

\subsection{Marginal distributions of submatrices}
In this subsection we go beyond marginal distributions and investigate the asymptotic distribution of sub-arrays of the matrix, in particular showing that for boxes of sidelength almost $\sqrt{n}$ the entries are close to iid exponentials after rescaling.

Fix some $k=k(n)=O(\frac{n^{1/2}}{\log n})$.  Define $W^{\ell_1 \ell_2}\in \mathbb{R}^{k^2}$ as the $k\times k$-submatrix of entries of the matrix $Y_{ij}$ for $i\in\{(\ell_1-1)k+1,\ldots,\ell_1 k \}$ and $j\in\{(\ell_2-1)k+1,\ldots,\ell_2 k \}$, i.e.,
\[
W^{\ell_1\ell_2}=\left(
         \begin{array}{ccc}
           Y_{(\ell_1-1)k+1,(\ell_2-1)k+1} & \ldots & Y_{(\ell_1-1)k+1,\ell_2k} \\
           \vdots & \ddots & \vdots \\
           Y_{\ell_1k,(\ell_2-1)k+1} & \ldots & Y_{\ell_1k,\ell_2k} \\
         \end{array}
       \right)~ .
\]
Let $\epsilon>0$ and let $A$ be a measurable subset of $R^{k^2}$.  By the Azuma–--Hoeffding inequality  we have the following large deviations bound.
\begin{equation}\label{e:boxesLDP}
P\left(\left|\left\lfloor\frac{n-1}{k}\right\rfloor^{-2} \sum_{\ell_1=1}^{\lfloor n-1/k\rfloor} \sum_{\ell_2=1}^{\lfloor n-1/k\rfloor} I(W^{\ell_1\ell_2}\in A) - P(W^{1 1} \in A) \right| > \frac12 \epsilon \right) \leq \exp\left(-\frac{\epsilon^2}8 \left\lfloor\frac{n-1}{k}\right\rfloor^{2}\right).
\end{equation}

Now define $V^{\ell_1 \ell_2}\in \mathbb{R}^{k^2}$ as the $k\times k$-submatrix of  $X_{ij}$ with $i\in\{(\ell_1-1)k+1,\ldots,\ell_1 k \}$ and $j\in\{(\ell_2-1)k+1,\ldots,\ell_2 k \}$, i.e.,
\[
V^{\ell_1\ell_2}=\left(
         \begin{array}{ccc}
           X_{(\ell_1-1)k+1,(\ell_2-1)k+1} & \ldots & X_{(\ell_1-1)k+1,\ell_2k} \\
           \vdots & \ddots & \vdots \\
           X_{\ell_1k,(\ell_2-1)k+1} & \ldots & X_{\ell_1k,\ell_2k} \\
         \end{array}
       \right)~ .
\]
We now prove Theorem~\ref{t:convergenceBox} showing that $d_{\mathrm{tv}}(n V^{11},W^{11})$ converges to 0.

\begin{proof}[Proof of Theorem~\ref{t:convergenceBox}]
By equation \eqref{e:relativeProbabilities} and \eqref{e:boxesLDP} we have that,
\[
P\left( \left|\frac1{n(n-1)} \sum_{i=1}^{n-1}\sum_{j=1}^{n-1} I(n V^{ij} \in \mathcal{A} ) - P(W^{11}\in \mathcal{A})  \right | > \frac12 \epsilon \right) \leq n^{4n} \exp\left(-\frac{\epsilon^2}8 \left\lfloor\frac{n-1}{k}\right\rfloor^{2}\right) = o(1).
\]
Since the entries of $X$ are exchangeable this implies that,
\[
 \left|P(n V^{11} \in \mathcal{A} ) - P(W^{11}\in\mathcal{A}) \right | < \frac12\epsilon  +  o(1).
\]
As this holds uniformly over all $\mathcal{A}$ it follows that $d_{\mathrm{tv}}(n V^{11},W^{11})< \epsilon$ for large $n$ which establishes the result.

\end{proof}

\section{Further properties of uniform doubly stochastic matrices}\label{s:other}
In this section we establish further properties of the matrices including convergence of moments and the mixing time of such matrices.

\subsection{Non-constant row sums}\label{s:nonConstant}
It will be important to consider the generalized case of $m\times n$-matrices with fixed but non-constant row and column sums.
For a sequence of positive row sums $\{a_i\}_{i=1}^m$ and columns sums $\{b_i\}_{i=1}^n$ where $\sum_{i=1}^m a_i = \sum_{i=1}^n b_i=t$ we define the transportation polytope $\mathfrak{p}=\mathfrak{p}\left((a_i),(b_i)\right)$ to be the polytope of $m\times n$-matrices with nonnegative entries, row sums $a_i$ and column sums $b_i$.  Let $\mathcal{P}_{m,n,t}$ denote the set of all such polytopes and let $\mathfrak{p}^*=\mathfrak{p}^*_{m,n,t}$ denote the special case of polytopes with constant row sums $t/m$ and column sums $t/n$.  We will let $\hbox{Vol}_{(m-1)(n-1)}(\mathfrak{p})$ denote the volume of the image of the set $\mathfrak{p}$ under the map
\[
(X_{ij})_{i=1,\ldots,m,j=1,\ldots,n}\mapsto(X_{ij})_{i=1,\ldots,m-1,j=1,\ldots,n-1}
\]
in $\mathbb{R}^{(m-1)(n-1)}$.
The following lemma shows that amongst all $m\times n$-matrices $\mathfrak{p}^*$ has the largest volume.

\begin{lemma}\label{l:constantPolytope}
We have that
\[
\hbox{Vol}_{(m-1)(n-1)}(\mathfrak{p}^*_{m,n,t})=\max_{\mathfrak{p} \in \mathcal{P}_{m,n,t}} \hbox{Vol}_{(m-1)(n-1)}(\mathfrak{p})
\]
\end{lemma}

\begin{proof}
We begin by proving the following simpler claim.
\begin{claim}\label{c:twoColumn}
Let $\{a_i\}_{i=1}^m$ be a collection of row sums with $\sum_{i=1}^m a_i =t$ and let $\mathfrak{p}(r)$ denote the polytope of $m\times 2$-matrices with row sums $(a_i)$ and column sums $r,t-r$ for $0\leq r \leq t$.  Then
\[
\hbox{Vol}_{(m-1)} \mathfrak{p}(t/2) = \max_{0\leq r \leq t} \hbox{Vol}_{(m-1)} \mathfrak{p}(r).
\]
\end{claim}

Let $X=(X_{ij})_{i=1,\ldots,m,j=1,2}$ be chosen uniformly according to $\mathfrak{p}(r)$.  Let $(Y_i)_{i=1\ldots,m}$ be independent random variables with the uniform distribution $[0,a_i]$.  It is easy to verify that $(X_{i1})_{i=1\ldots,m}$ is equal in distribution to $(Y_i)_{i=1\ldots,m}$ conditional on $\sum_{i=1}^m Y_i =r$ and moreover that the volume $\hbox{Vol}_{(m-1)} \mathfrak{p}(r)$ is proportional to the density of $\sum_{i=1}^m Y_i$ at $r$.

It remains to show that this density is maximized at $t=r/2$.  We say a distribution is log-concave if the logarithm of its density concave.  This clearly includes the uniform distribution on an interval.  Moreover, the sum of independent random variables with log-concave distributions itself has a log-concave distribution \cite{boydVan:04}.  Since the density of $\sum_{i=1}^m Y_i$ is symmetric about $t/2$ it follows that it is maximized   at $t/2$ which completes the claim.

We now complete the proof of Lemma \ref{l:constantPolytope}.
Let $\mathfrak{p}=\mathfrak{p}\left((a_i),(b_i)\right)$ and $\mathfrak{p}'=\mathfrak{p}\left((a_i),(b_i')\right)$ where $b_1'=b_2'=\frac{b_1+b_2}{2}$ and $b_i'=b_i$ for $i\geq 3$.  Further define the set
\[
\hat{\mathcal{A}}=\{(\hat a_i)_{i=3}^m:0\leq \hat a_i\leq a_i,\sum_{i=1}^m \hat a_i = t-b_1-b_2\}
\]
which represent possible values for the sum of the entries of the rows of a matrix in $\mathfrak{p}$ excluding the first two columns.  Then by first conditioning on these sums we have the following integral for the volumes
\begin{align*}
&\hbox{Vol}_{(m-1)(n-1)} \mathfrak{p}\left((a_i),(b_i)\right)\\
&\qquad = \int_{\hat{\mathcal{A}}}
\hbox{Vol}_{(m-1)} \mathfrak{p}\left((a_i-\hat a_i),(b_i)_{i=1,2}\right)
\hbox{Vol}_{(m-1)(n-3)} \mathfrak{p}\left((\hat a_i),(b_i)_{i=3,\ldots,n}\right) \mu(d (\hat a_i))
\end{align*}
where $\mu$ is the uniform distribution over $\mathcal{A}^*$.  Similarly
\begin{align*}
&\hbox{Vol}_{(m-1)(n-1)} \mathfrak{p}\left((a_i),(b_i')\right)\\
&\qquad = \int_{\hat{\mathcal{A}}}
\hbox{Vol}_{(m-1)} \mathfrak{p}\left((a_i-\hat a_i),(b_i')_{i=1,2}\right)
\hbox{Vol}_{(m-1)(n-3)} \mathfrak{p}\left((\hat a_i),(b_i')_{i=3,\ldots,n}\right) \mu(d (\hat a_i))
\end{align*}
Applying Claim \ref{c:twoColumn} we, therefore, have that
\[
\hbox{Vol}_{(m-1)(n-1)} \mathfrak{p}\left((a_i),(b_i)\right) \leq \hbox{Vol}_{(m-1)(n-1)} \mathfrak{p}\left((a_i),(b_i')\right)
\]
which says that replacing the first two column sums by their average can only increase the volume of the polytope.  This is true of course for any pair of columns and similarly for any pair of rows.
It is easy to show that the volume of polytopes in $\mathcal{P}_{m,n,t}$ are symmetric and continuous  in the row and column sums $(a_i),(b_i)$ and hence it follows that  $\mathfrak{p^*}$ must be a maxima of the volume.
\end{proof}

Canfield and McKay \cite{CanMcK:07} give an asymptotic formula for the volume of matrices with constant row and column sums as
\begin{align}\label{e:volMN}
&\hbox{Vol}_{(m-1)(n-1)}(\mathfrak{p}^*_{m,n,m})\nonumber\\
&\qquad= \frac1{m^{(n-1)/2}n^{(m-1)/2}} \cdot \frac1{(2\pi)^{(m+n-1)/2}n^{(m-1)(n-1)}}\exp\left(\frac13+mn-\frac{(m-n)^2}{12mn}+o(1)\right).
\end{align}
Note that our definition of volume corresponds to their notion of volume in units of basic cells of the lattice induced by $\mathbb{Z}^{mn}$.

Let $\mathcal{R}=\mathcal{R}_{r,n}$ denote the $r(n-1)$-dimensional polytope of nonnegative matrices whose rows sum to 1.  Let $\nu_r$ denote the measure on $\mathcal{R}$ induced by the first $r$ rows of a uniform doubly stochastic $n\times n$-matrix $(X_{ij})$ and let $\mu_r$ denote uniform probability measure on $\mathcal{R}$.  Equivalently $\mu_r$ is the measure induced by the first $r$ rows of a uniform stochastic matrix(one where the rows are independent and conditioned to sum to 1).

\begin{lemma}\label{l:radonBound}
For a fixed integer $r\geq 1$ and $n>r$ the Radon-Nikodym derivative of the measures $\mu_r$ and $\nu_r$ satisfies
\[
\frac{d \nu_r}{d \mu_r} \leq (1+o(1))e^{r/2} .
\]
as $n\to\infty$.
\end{lemma}
\begin{proof}
Conditioned on the first $r$ rows of a uniform doubly stochastic $n\times n$-matrix $(X_{ij})$ the remainder of the matrix is a uniformly chosen matrix from the polytope of $(n-r)\times n$-matrices
\[
\mathfrak{p}\Big(1_{n-r},\big(1-\sum_{i=1}^r X_{ij}\big)_{j=1,\ldots,n}\Big)
\]
where $1_m$ represents the vectors of 1's of length $m$.  Since $\mu_r$ is the uniform distribution over $\mathcal{R}=\mathcal{R}_{r,n}$ it follows that
\[
\frac{d \nu_r}{d \mu_r}(X_{ij}) \propto \hbox{Vol}_{(n-r-1)(n-1)} \mathfrak{p}\Big(1_{n-r},\big(1-\sum_{i=1}^r X_{ij}\big)_{j=1,\ldots,n}\Big)
\]
where $\propto$ denote proportionality.  To determine the constant of proportionality note that
\[
Z_n= \hbox{Vol}_{r(n-1)}(\mathcal{R}) \int_{\mathcal{R}}  \hbox{Vol}_{(n-r-1)(n-1)} \mathfrak{p}\Big(1_{n-r},\big(1-\sum_{i=1}^r X_{ij}\big)_{j=1,\ldots,n}\Big) \mu_r(d (X_{ij}))
\]
recalling that $Z_n$ is the volume of the Birkhoff polytope.  It follows that
\begin{align*}
\frac{d \nu_r}{d \mu_r}(X_{ij})&= \frac1{Z_n}\hbox{Vol}_{r(n-1)}(\mathcal{R})  \hbox{Vol}_{(n-r-1)(n-1)} \mathfrak{p}\Big(1_{n-r},\big(1-\sum_{i=1}^r X_{ij}\big)_{j=1,\ldots,n}\Big) \\
&\leq \frac1{Z_n}\hbox{Vol}_{r(n-1)}(\mathcal{R})  \hbox{Vol}_{(n-r-1)(n-1)} \mathfrak{p}^*_{m,n,m}
\end{align*}
by Lemma \ref{l:constantPolytope}.  Hence substituting the formulas for the volumes of the polytopes and applying Stirling's formula we have that
\begin{align*}
\frac{d \nu_r}{d \mu_r}(X_{ij})&\leq (1+o(1))\frac{n^{n-1}}{(n-r)^{(n-1)/2}n^{(n-r-1)/2}} \cdot \frac1{((n-1)!)^r} \cdot \frac{(2\pi)^{n-1/2}n^{(n-1)^2}e^{-rn}}{(2\pi)^{(2n-r-1)/2}n^{(n-r-1)(n-1)}}\\
&= (1+o(1))n^{r/2}e^{r/2} \cdot \frac{n^r}{(\sqrt{2\pi n}\ n^n e^{-n})^r} \cdot (2\pi)^{r/2}n^{r(n-1)}e^{-rn}\\
&= (1+o(1))e^{r/2}
\end{align*}
which completes the proof.
\end{proof}

This proof also shows that the uniformly distributed stochastic matrix is not given exactly by the square of the absolute value of a random unitary matrix.  In such a random matrix the rows are distribution according to $\mu_1$ while we have that
\[
\frac{d \nu_1}{d \mu_1}(\tfrac1n 1_n) =\frac1{Z_n}\hbox{Vol}_{r(n-1)}(\mathcal{R})  \hbox{Vol}_{(n-r-1)(n-1)} \mathfrak{p}^*_{m,n,m}= (1+o(1))e^{r/2}.
\]
Hence at least for large $n$ the models are not the same (in the trivial case of $n=2$ they are equal).

\subsection{Convergence of Moments}\label{s:moments}
Using Lemma \ref{l:radonBound} we may now establish convergence of the moments of the entries of a doubly stochastic matrix to those of independent exponentials.  We will let $(V_k)$ be a sequence of iid exponentially distributed mean 1 random variables.

\begin{lemma}
Let $(i_1,j_1),\ldots,(i_L,j_L)$ be a fixed sequence of pairs of positive integers and $\alpha_1,\ldots,\alpha_L$ be fixed a sequence of positive integers.  Then if $(X_{ij})_{i,j=1,\ldots,n}$ are distributed as a uniform doubly stochastic matrix then
\[
E\prod_{k=1}^L (n X_{i_k,j_k})^{\alpha_k} \rightarrow E\prod_{k=1}^L V_{k}^{\alpha_k}.
\]
\end{lemma}

\begin{proof}
By Theorem~\ref{t:convergenceBox} the joint distribution of the $(n X_{i_k,j_k})_{k=1,\ldots,L}$ converges to iid exponential random variables.  It follows that
\[
E\Bigg[\prod_{k=1}^L (n X_{i_k,j_k})^{\alpha_k}I(\max_{1\leq k \leq L} n X_{i_k,j_k} < M )\Bigg] \rightarrow E\Bigg[\prod_{k=1}^L V_{k}^{\alpha_k} I(\max_{1\leq k \leq L} V_k < M )\Bigg]~,
\]
and hence we can complete the proof by showing that
\begin{equation}\label{e:momentsXLarge}
\lim_{M\to\infty}\limsup_n E\Bigg[\prod_{k=1}^L (n X_{i_k,j_k})^{\alpha_k}I(\max_{1\leq k \leq L} n X_{i_k,j_k} \geq M )\Bigg] \rightarrow 0.
\end{equation}
By the exchangeability of $X$ we may assume without loss of generality that $\max_{1\leq k \leq L} i_k \leq L$ and that $\max_{1\leq k \leq L} j_k \leq L$.  In particular this assumption implies that each of the entries $X_{i_k,j_k}$ appear in the first $L$ rows of the matrix.  Let $\widetilde{Y}_{ij}$ denote a uniform stochastic matrix, that is one whose rows are independent and chosen according to $\mu_1$.

Now by Lemma~\ref{l:radonBound} it follows that
\begin{align*}
&E\Bigg[\prod_{k=1}^L (n X_{i_k,j_k})^{\alpha_k}I(\max_{1\leq k \leq L} n X_{i_k,j_k} \geq M )\Bigg] \\
&\qquad\leq (e^{L/2}+o(1))
E\Bigg[\prod_{k=1}^L (n \widetilde{Y}_{i_k,j_k})^{\alpha_k}I(\max_{1\leq k \leq L} n \widetilde{Y}_{i_k,j_k} \geq M )\Bigg]
\end{align*}
and hence it is sufficient to establish equation~\eqref{e:momentsXLarge} replacing the $X_{i_k,j_k}$ with $\widetilde{Y}_{i_k,j_k}$.  Now the $Y_{i_k,j_k}$ are given by Beta distributions $B(1,n-1)$.  It follows that
\begin{equation}\label{e:BetaMoments}
E \widetilde{Y}_{i_k,j_k}^{\alpha_k} = \prod_{\ell=1}^{\alpha_k} \frac{\ell}{n-1+\ell} = (1+o(1))\alpha_k! n^{-\alpha_k}
\end{equation}
By the power mean inequality and the fact that $E |\widetilde{Y}|^\alpha I(Y>M) \leq M^{-1}E|\widetilde{Y}|^{\alpha+1}$
\begin{align*}
&E\prod_{k=1}^L (n \widetilde{Y}_{i_k,j_k})^{\alpha_k} I(\max_{1\leq k \leq L} n \widetilde{Y}_{i_k,j_k} \geq M )\\
&\qquad\leq E\frac1{\sum_{k=1}^L \alpha_k}\sum_{k=1}^L \alpha_k (n  \widetilde{Y}_{i_k,j_k})^{\sum_{k=1}^L \alpha_k} I(\max_{1\leq k \leq L} n \widetilde{Y}_{i_k,j_k} \geq M )\\
&\qquad\leq M^{-1} E\frac1{\sum_{k=1}^L \alpha_k}\sum_{k=1}^L \alpha_k (n \widetilde{Y}_{i_k,j_k})^{1+\sum_{k=1}^L \alpha_k}
\end{align*}
and hence by equation \eqref{e:BetaMoments},
\[
\lim_{M\to\infty}\sup_n E\prod_{k=1}^L (n Y_{i_k,j_k})^{\alpha_k} I(\max_{1\leq k \leq L} n Y_{i_k,j_k} \geq M )=0
\]
which completes the proof.
\end{proof}

We may also examine the maximal element of the matrix.  For an $n\times n$-matrix of iid exponential random variables with mean 1 the maximum entry is at most $(2+o(1))\log n$ with high probability and we show that this is also the case for the renormalized uniform doubly stochastic matrix.

\begin{proof}[Proof of Theorem~\ref{t:maxEntry}]
By Lemma~\ref{l:radonBound} we have that
\[
P(n X_{11} > (2+\epsilon)\log n) \leq (e^{1/2}+o(1)) P(n Y_{11} > (2+\epsilon)\log n)
\]
Now since $Y_{11}$ has $B(1,n-1)$ distribution
\begin{align}\label{e:entryTailBound}
P(n Y_{11} > (2+\epsilon)\log n) &= (n-1) \int_{\frac{(2+\epsilon)\log n}{n}}^1 (1-y)^{n-2}\nonumber\\
&=\left(1-\frac{(2+\epsilon)\log n}{n}\right)^{n-1}\nonumber\\
&=(1+o(1))n^{-2-\epsilon}.
\end{align}
The exchangeability of the entries and a union bound completes the proof.
\end{proof}

\subsection{Mixing Time}
As uniformly distributed stochastic matrices correspond to the transition matrices of Markov chains one can ask about the mixing time of such matrices.

\begin{proof}[Proof of Theorem~\ref{t:mixing}]
By Lemma~\ref{l:radonBound} the mixing time cannot be 1 since it implies that the rows of the matrix are not close to being constant.  We show at time 2, however, they are almost constant.  Let $X^{(2)}_{ij}$ denote the $ij$-th entry of the matrix $X^2$. The total variation distance from stationarity of the Markov chain at time 2 is given by
\begin{align*}
\max_i \frac12\sum_{j=1}^n \left|\frac1n - X^{(2)}_{ij} \right|
\end{align*}
which is equal to
\begin{align*}
\max_i \sum_{j=1}^n \max\{ \frac1n - X^{(2)}_{ij},0 \} .
\end{align*}
Since the rows are exchangeable, by taking a union bound it is sufficient to show that for each $\epsilon>0$,
\[
P\left( \sum_{j=1}^n \max\{ \frac1n - X^{(2)}_{1j},0\} > \epsilon\right) = o(1/n).
\]
We will again work first in the independent entries model $(Y_{ij})$. Let $\mathcal{F}$ denote the $\sigma$-algebra generated by $(Y_{1j})_{j=1,\ldots,n-1}$ and let $\mathcal{H}$ denote the event
\[
\left\{\max_{1\leq j \leq n-1} Y_{1j} \leq 3 \log n\right\}\cap\left\{ \sum_{j=1}^{n-1} Y_{1j} \leq n-3\log n \right\}.
\]
The sums $\sum_{k=2}^{n-1} Y_{1k} Y_{kj}$ are conditionally independent given $\mathcal{F}$.  Further for $\delta,\lambda>0$,
\begin{align*}
&P\left( \frac1n \sum_{k=2}^{n-1} Y_{1k} Y_{kj} < 1-\delta \text{ and } \mathcal{H}\mid \mathcal{F} \right) = P\left(\frac1n\sum_{k=2}^{n-1} Y_{1k} (1-Y_{kj}) > \delta-6\log n \text{ and } \mathcal{H}\mid \mathcal{F} \right)\\
&= P\left( \exp\left( \frac{\lambda}{n} \sum_{k=2}^{n-1} Y_{1k} (1-Y_{kj})\right) > \exp\left(\tfrac{\lambda}{n}(\delta-6\log n)\right) \text{ and } \mathcal{H}\mid \mathcal{F} \right)
\end{align*}
Now if $Y_{1k}\leq\frac3{n}\log n$ and $\lambda =\frac{n}{\log^3 n}$ then by Taylor series for large $n$ and $1\leq j \leq n-1$,
\[
E \left[\exp\left(\tfrac{\lambda}{n} Y_{1k} (1-Y_{kj}) \right)\mid Y_{1k}\right] = \frac{\exp(\tfrac{\lambda}{n} Y_{1k})}{1+\tfrac{\lambda}{n} Y_{1k}}\leq \frac{1+\tfrac{\lambda}{n} Y_{1k}+ (\tfrac{\lambda}{n} Y_{1k})^2 }{1+\tfrac{\lambda}{n} Y_{1k}} \leq \exp\left((\tfrac{\lambda}{n} Y_{1k})^2\right).
\]
Hence by Markov's inequality for large $n$,
\begin{align*}
P\left( \frac1n \sum_{k=2}^{n-1} Y_{1k} Y_{kj} < 1-\delta \text{ and } \mathcal{H}\mid \mathcal{F} \right) &\leq \frac{\exp\left( \frac{9n}{\log^4 n} \right)}
{\exp\left( \frac{n(\delta-\frac6{n}\log n)}{\log^3 n} \right)} \leq \exp\left( - \frac{n\delta}{2\log^3 n} \right)
\end{align*}
with room to spare.  By the conditional independence of the sums we have that
\begin{align}\label{e:sumYVBound}
P\left(\#\left\{ 1\leq j \leq n-1: \frac1n \sum_{k=2}^{n-1} Y_{1k} Y_{kj} < 1-\delta\right\}>\delta n \text{ and } \mathcal{H}\mid \mathcal{F} \right) \leq {n \choose n\delta} \exp\left( - \frac{n^2\delta^2}{2\log^3 n} \right).
\end{align}
This implies that
\[
P\left(\sum_{j=1}^{n-1} \max\left\{ \frac1n - \sum_{k=2}^{n-1} \frac{Y_{1k}}{n} \frac{Y_{kj}}{n} ,0\right\} > 2\delta \text{ and } \mathcal{H} \right)  \leq {n \choose n\delta} \exp\left( - \frac{n^2\delta^2}{2\log^3 n} \right).
\]
We can now return to the doubly stochastic matrix setting.  By equation \eqref{e:relativeProbabilities} we have that
\[
P\left(\sum_{j=1}^{n-1} \max\left\{ \frac1n - \sum_{k=2}^{n-1} X_{1k} X_{kj} ,0\right\} > 2\delta ,\max_{1\leq j \leq n} X_{1n}\geq \frac{3\log n}{n}\right)  \leq n^{4n}{n \choose n\delta} \exp\left( - \frac{n^2\delta^2}{2\log^3 n} \right).
\]
and hence since $X^{(2)}_{1j} = \sum_{k=1}^{n} X_{1k} X_{kj} \geq \sum_{k=2}^{n-1} X_{1k} X_{kj}$ and so
\[
P\left(\sum_{j=1}^{n} \max\left\{ \frac1n - X^{(2)}_{1j} ,0\right\} > 2\delta +\frac1n ,\max_{1\leq j \leq n} X_{1n}\geq \frac{3\log n}{n}\right) =o(1/n)).
\]
By equation \eqref{e:entryTailBound} we have that
\[
P(\max_{1\leq j \leq n} X_{1n}\geq \frac{3\log n}{n}) = O(n^{-2})
\]
so it follows that
\[
P\left( \sum_{j=1}^n \max\{ \frac1n - X^{(2)}_{1j},0\} > 2\delta + \frac1n\right) = o(1/n)
\]
for any $\delta>0$.  Letting $\delta$ go to 0 completes the proof.
\end{proof}

\section{Singular Values}\label{s:singular}

In this section we give the proof of Theorem~\ref{t:singular}.  Let $0\leq\sigma_1^n\leq \dots \leq \sigma_n^n$ denote the singular values of $n^{1/2}(X-EX)$.   These correspond to the square roots of the eigenvalues of the matrix $n(X-EX)(X-EX)^*$ which is a Hermitian matrix.  For a Hermitian matrix $A$ let $\lambda_1(A) \leq \ldots\leq \lambda_n(A)$ denote its eigenvalues and let $\hmu(A)=\sum_{i=1}^n \delta_{\lambda_i(A)}$ denote the empirical spectrum of $A$.

Let $(\widetilde Y_{ij})_{i,j=1,\ldots,n}$ denote the $n\times n$-matrix with i.i.d. entries supported in $[0,K]$ and consider the Wishart Matrix $\Xi_n=n^{-1}(\widetilde Y-E\widetilde Y)(\widetilde Y-E \widetilde Y)^*$ which is Hermitian and hence has real eigenvalues.  Mar{\v{c}}enko and Pastur~\cite{MarPas:67} showed that $\hmu(\Xi_n)\to\mu'$
weakly in probability as $n\to\infty$ where $\mu'$ is the distribution on $[0,2]$ with density $\frac{\sqrt{x(4-x)}}{2\pi x}$.

As with our previous results we use large deviation results on random matrices to transfer results to uniform doubly stochastic matrices. In this case we use results of Guionnet and Zeitouni~\cite{GuiZei:00} who establish concentration of measure results for the spectrum of large Wishart matrices.  In Corollary~1.8 and the remarks that follow they show that for any $\epsilon>0$ there exists $c(\epsilon)>0$ such that for large $n$ and $K>1$,
\begin{equation}\label{e:GWconcentration}
P\left( d_{W}\left(\hmu(\Xi_n), E \hmu(\Xi_n) \right) > \epsilon\right) \leq \exp\left( -c K^{-2} n^2 \right).
\end{equation}
where $d_{W}$ denotes the Wasserstein distance.  We will take the entries of $\widetilde Y$ to have density given by
\begin{equation}\label{e:truncatedDist}
\rho_n(x)=\begin{cases} \frac1{1-n^{10}} \ e^{-x}\quad & x\in[0,10 \log n] \, ,\\
0& \hbox{o.w.} \end{cases}
\end{equation}
That is the entries are mean 1 exponentials conditioned to be less than $10\log n$ and so it follows that
\begin{equation}\label{e:GWconcentration}
P\left( d_{W}\left(\hmu(\Xi_n), E \hmu(\Xi_n) \right) > \epsilon\right) \leq \exp\left( -c' n^2 \log^{-2} \right).
\end{equation}

Now let
\[
\tilde{S}_n=\left\{(x_{ij})_{i,j=1,\ldots,n-1} \in S_n: \max_{1\leq i,j \leq n} \Gamma(x)_{ij} \leq \tfrac{6}{n}\log n \right \}
\]
which corresponds to the doubly stochastic matrices whose maximum entry is at most $\frac{6}{n}\log n$.
Also define
\[
\widetilde{\mathcal{D}_n} = \left\{(x_{ij})_{i,j=1,\ldots,n} \in [0,8\log n]^{n^2}: \tfrac1n(x_{ij})_{i,j=1,\ldots,n-1} \in \tilde{S}_n, \forall 1\leq  i,j\leq n, 0\leq (x - \Gamma(x))_{ij}\leq n^{-4} \right \}.
\]
The following lemma is the analogue of Lemma~\ref{l:ProbDn} for $\widetilde Y$.
\begin{lemma}\label{l:ProbTildeDn}
With $\widetilde Y$ as above with marginals given by \eqref{e:truncatedDist}, conditional on $\widetilde Y\in \widetilde{\mathcal{D}}_n$ we have that $\Gamma(\tfrac1n Y)$ is uniform on $\tilde{S}_n$.
Further, for large $n$ we have that,
\begin{equation}
P(\widetilde Y\in \widetilde{\mathcal{D}}_n) \geq n^{-8n}.
\end{equation}
\end{lemma}

\begin{proof}
Let $\mathcal{W}$ be the product of the intervals $\mathcal{W}= \prod_{1\leq i,j \leq n} I_{ij}$ where
\[
I_{ij}= \begin{cases} [0,n^{-4}] &\ \hbox{if } \max\{i,j\}=n\\
\{0\} & \ \hbox{o.w.}
\end{cases}
\]
Then for each fixed $\by \in \tilde{S}_n$ the set $\{ \widetilde Y\in \widetilde{\mathcal{D}}_n: \Phi(\tfrac1n \widetilde Y) = \by\}$ is  $n\Gamma(\by)+\mathcal{W}$.  Since the density of $\widetilde Y$ depends only on $\sum_{ij} \widetilde Y_{ij}$ and since $\sum_{ij} n\Gamma(\by)_{ij} \equiv n^2$ it follows that $\Gamma(\tfrac1n \widetilde Y)$ is uniform on $\tilde{S}_n$.

Now
\begin{align}\label{e:probTildeDn}
P(\widetilde Y\in \widetilde{\mathcal{D}}_n) &= (1-n^{-10})^{-n^2}\int_{\widetilde{\mathcal{D}}_n} \exp\left(-\sum_{i=1}^n \sum_{j=1}^n y_{ij} \right) dy_{11}\ldots dy_{nn}\nonumber\\
&=(1+o(1)) \exp(-n^2) n^{n^2} \hbox{Vol}_{n^2}(\mathcal{\mathcal{D}}_n)
\end{align}
as for all $\widetilde Y\in \widetilde{\mathcal{D}}_n$ we have that
\[
n^2\leq \sum_{i=1}^n \sum_{j-1}^n \widetilde Y_{ij}\leq n^2+ (2n+1)n^{-4}.
\]
The volume of $\mathcal{W}$ is clearly $n^{-4(2n-1)}$ so we have that
\[
\hbox{Vol}_{n^2}(\widetilde{\mathcal{D}}_n) = \hbox{Vol}_{(n-1)^2}(\tilde{S}_n) n^{-4(2n-1)}.
\]
Now interpreting $\tilde{S}_n$ as a subset of $S_n$ it corresponds to the set of doubly stochastic matrices whose maximum entry is at most $6\log n$.  Hence by Theorem~\ref{t:maxEntry} we have that
\begin{equation}\label{e:maxEntryRation}
\frac{\hbox{Vol}_{(n-1)^2}(\tilde{S}_n)}{\hbox{Vol}_{(n-1)^2}(S_n)}=P(\max_{ij} X_{ij} \leq 6\log n)=1-o(1).
\end{equation}
Combining equations \eqref{e:ZnAsymptotic}, \eqref{e:probDn}, \eqref{e:maxEntryRation} we have that
\begin{align}
P(\widetilde Y\in \widetilde{\mathcal{D}}_n) &=(1+o(1))   \frac{\exp(-n^2) n^{n^2} n^{-4(2n-1)}}{n^{n-1} (2\pi)^{n-1/2}n^{(n-1)^2}}\exp\Big(\frac13+n^2\Big)\nonumber\\
& \geq n^{-8n}
\end{align}
for large $n$.
\end{proof}

Now the Courant-Fischer Minimax Theorem says that for an $n\times n$ Hermitian matrix $X$  the $k$-th eigenvalue of $X$ is given by
\[
\lambda_k(X)= \min_{U:\hbox{dim}(U)=k} \max_{x\in U} \frac{x^* X x}{x^* x}
\]
where the minimum is over all $k$-dimensional subspaces of $\mathbbm{R}^n$.  It follows that for Hermitian matrices $X,Y$  that
\[
|\lambda_k(X) - \lambda_k(Y)| \leq \| X - Y \|_{\mathrm{op}} \leq n\max_{ij} |X_{ij}-Y_{ij}|
\]
where $\|\cdot\|_{\mathrm{op}}$ is the operator norm (see e.g. \cite{HorJoh:90}).  For $\by \in \tilde{S}_n$ and $\widetilde Y\in \widetilde{\mathcal{D}}_n$ such that $\Gamma(\tfrac1n \widetilde Y)=\Gamma(\by)$ we compare the eigenvalues of the matrices
\begin{align*}
A &=n(\Gamma(\by)-n^{-1}\mathbf{1})(\Gamma(\by)-\tfrac1n\mathbf{1})^*\\
B &=n^{-1}(\widetilde Y-y\mathbf{1})(\widetilde Y-y\mathbf{1})^*
\end{align*}
where $y=\frac{1-(1+10\log n)n^{-10}}{1-n^{-10}}=E\widetilde Y_{11}$ and $\mathbf{1}$ is the $n\times n$-matrix of all 1's.
By the above bound we have that for $1\leq k \leq n$,
\begin{align}\label{e:eignDiff}
\left|\lambda_k(A) - \lambda_k(B)\right| \leq n \max_{i,j} \left|A_{ij} - B_{ij}\right|
\end{align}
Breaking $A-B$ into parts we first have that
\begin{align}\label{e:singularDiffBoundA}
\sup_{i,j}\left|\left(n \Gamma(\by)^2 - n^{-1}\widetilde Y^2 \right)_{ij} \right|
&= \sup_{i,j} n^{-1} \left|\left( 2(n\Gamma(\by))(\widetilde Y-n\Gamma(\by)) + \left(\widetilde Y-n\Gamma(\by)\right)^2 \right)_{ij}\right|\nonumber\\
& = O(n^{-3})
\end{align}
since $\max_{i,j} (n\Gamma(\by))_{ij} \leq 6\log n$ and $\max_{i,j} |(\widetilde Y-n\Gamma(\by))_{ij}| \leq n^{-4}$.  Also
\begin{align}\label{e:singularDiffBoundB}
\sup_{i,j}\left|\left(n \Gamma(\by)\cdot n^{-1}\mathbf{1} - n^{-1}\widetilde Y \cdot y\mathbf{1}\right)_{ij}\right|
&= \sup_{i,j}\left|\left(\left(\Gamma(\by) - n^{-1} y \widetilde Y\right)\mathbf{1}\right)_{ij}\right|\nonumber\\
&= \sup_{i,j}\left|\left(\left(\Gamma(\by) - n^{-1}  \widetilde Y\right)\mathbf{1}\right)_{ij}\right| +O(n^{-10})\nonumber\\
& = O(n^{-4})
\end{align}
since $1-y=O(n^{-10})$.  Finally we have that
\begin{align}\label{e:singularDiffBoundC}
\sup_{i,j}\left|\left(n^{-1} \mathbf{1} - n^{-1}y^2\mathbf{1}\right)_{ij}\right| = O(n^{-10})
\end{align}
since $1-y^2=O(n^{-10})$.  Combining \eqref{e:eignDiff}, \eqref{e:singularDiffBoundA}, \eqref{e:singularDiffBoundB} and \eqref{e:singularDiffBoundC} it follows that
\begin{align}\label{e:eignDiff}
\left|\lambda_k(A) - \lambda_k(B)\right| \leq O(n^{-2}).
\end{align}
In particular we have that for large $n$ if $d_{W}(\hmu(A),\hmu(B))= o(1)$ uniformly in $\by$ and $\widetilde Y$. With $\Xi_n$ defined above and $X$ a uniform doubly stochastic matrix by Lemma~\ref{l:ProbTildeDn} we have that for any $\epsilon>0$ and large enough $n$ that
\begin{align}
&P\left( d_{W}\left(\hmu(n(X-EX)(X-EX)^*), E\hmu(\Xi_n)\right)>2\epsilon\mid \Phi(X)\in\tilde{S}_n\right)\nonumber\\
 &\qquad \leq P\left( d_{W}\left(\hmu(\Xi_n), E\hmu(\Xi_n)\right)>\epsilon\mid \widetilde Y\in \widetilde{\mathcal{D}}_n \right)\nonumber\\
 &\qquad\leq  P\left( d_{W}\left(\hmu(\Xi_n), E\hmu(\Xi_n)\right)>\epsilon\right) P\left(\widetilde  Y\in \widetilde{\mathcal{D}}_n \right)^{-1}\nonumber\\
 &\qquad\leq n^{8n}\exp\left( -c' n^2 \log^{-2} \right) = o(1)
\end{align}
where the final inequality follows from Lemma~\ref{l:ProbTildeDn} and equation~\eqref{e:GWconcentration}. Now by Theorem~\ref{t:maxEntry},
\[
P\left(  \Phi(X)\in\tilde{S}_n\right) \to 1
\]
so
\[
P\left( d_{W}\left(\hmu(n(X-EX)(X-EX)^*), E\hmu(\Xi_n)\right)>2\epsilon\right)\to 0
\]
as $n\to\infty$.  As $E\hmu(\Xi_n)\to \mu'$ (see e.g. \cite{MarPas:67,AGZ:09}) it follows that
\[
\hmu(n(X-EX)(X-EX)^*) \to \mu'
\]
weakly in probability as $n\to\infty$.  Since the singular values of $n^{1/2}(X-EX)$ are the positive square roots of the eigenvalues of $n^{1/2}(X-EX)(X-EX)^*$ and the map $x\mapsto x^2$ maps $\mu$ to $\mu'$ this completes the proof of Theorem~\ref{t:singular}.

\bibliographystyle{plain}
\bibliography{all}

\end{document}